\documentclass[secthm]{elsart}
\usepackage{amsmath,latexsym,epsfig}
\usepackage{amssymb,epic,eepic}
\usepackage[english]{babel}
\usepackage{graphicx}
\usepackage{comment}
\newcommand{\ml}{l\kern-0.035cm\char39\kern-0.03cm}
\newcommand{\md}{d\kern-0.035cm\char39\kern-0.03cm}
\newcommand{\mt}{t\kern-0.035cm\char39\kern-0.03cm}

\newtheorem{lemma}{Lemma}[section]

\newtheorem{corollary}[lemma]{Corollary}
\newtheorem{theorem}[lemma]{Theorem}

\newtheorem{definition}[lemma]{Definition}
\newtheorem{statement}[lemma]{Statement}
\newtheorem{example}[lemma]{Example}
\newtheorem{problem}[lemma]{Problem}
\newtheorem{remark}[lemma]{Remark}

\def\cH{{{\mathcal H}}}
\def\clos#1{\overline{#1}}

\newcommand{\soucin}[2]{\langle#1,#2\rangle}

\renewcommand{\vec}[1]{{\mathbf{#1}}}
\theoremstyle{plain}
\newenvironment{proof}[1][Proof]{\textbf{#1.} }{\ \rule{0.5em}{0.5em}}
\def\logoesf{
\begin{tabular}{l l}
\begin{tabular}{c}
{Supported by}\\
\phantom{\huge X}
\end{tabular}& \ \resizebox{8.58cm}{!}{\includegraphics{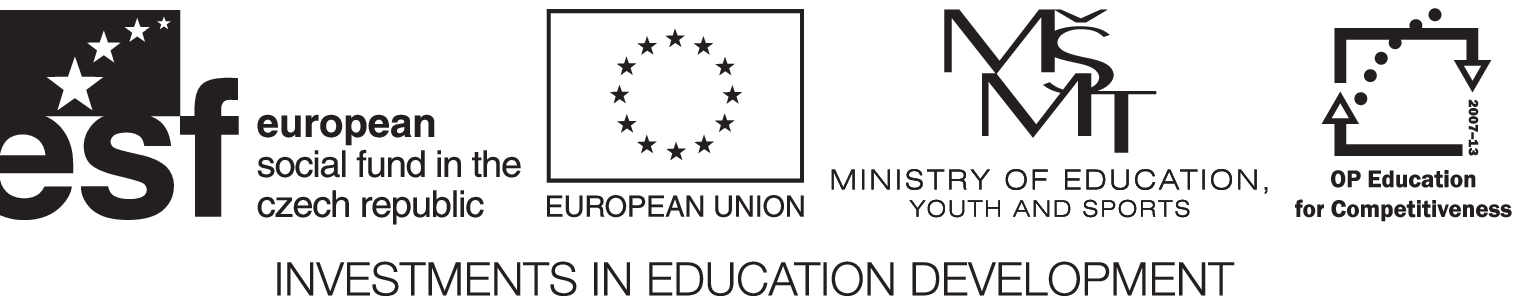}}
\end{tabular}}

\begin{document}
\begin{frontmatter}
\title{On realization of generalized effect algebras}

\author[Paseka]{Jan Paseka\corauthref{cor}\thanksref{wow}}
\corauth[cor]{Corresponding author.}
\ead{paseka@math.muni.cz}
\thanks[wow]{The author acknowledges the support by ESF Project CZ.1.07/2.3.00/20.0051
Algebraic methods in Quantum Logic of the Masaryk University.\\
\logoesf%
}

\address[Paseka]{Department of Mathematics
 and Statistics,
Faculty of Science,
Masaryk University, 
Kotl\'a\v{r}sk\'a~2,
CZ-611~37~Brno, 
Czech Republic}

\date{\today}





\begin{abstract}
A well known fact is that there is a finite orthomodular lattice with an order 
 determining set of states which is not representable in the standard 
quantum logic, the lattice $L({\mathcal H})$ of all closed subspaces of 
a separable complex Hilbert space. 

We show that a generalized effect algebra is  representable in the operator 
generalized effect algebra 
${\mathcal G}_D({\mathcal H})$ of effects  of a complex Hilbert space ${\mathcal H}$ iff 
it has  an order determining set of generalized states.
This extends the corresponding results for 
effect algebras of Rie\v canov\'a and Zajac. Further, any operator 
generalized effect algebra  
${\mathcal G}_D({\mathcal H})$ possesses
an order determining set of generalized states

\end{abstract}

\end{frontmatter}

\section{Introduction}

In \cite{greechie} Greechie proved that there is a finite orthomodular lattice with an order 
 determining set of states which is not  representable in the standard 
quantum logic, the lattice $L({\mathcal H})$ of all closed subspaces of 
a separable complex Hilbert space. 

In \cite{pasekariec} the following problem was formulated. 

{\renewcommand{\labelenumi}{{\normalfont  (\roman{enumi})}}
\begin{problem}\label{problem} Find 
\begin{enumerate}
\settowidth{\leftmargin}{(iiiii)}
\settowidth{\labelwidth}{(iii)}
\settowidth{\itemindent}{(ii)}
\item necessary conditions
\item sufficient conditions
\item necessary  and sufficient conditions
\end{enumerate}
for an (finite, lattice-ordered, MV-effect algebra respectively) effect algebra $E$ 
that  there is a complex Hilbert space  ${\mathcal H}$  and linear 
 operators densely defined in ${\mathcal H}$  as elements of $E$ such that,  
 whenever the effect algebraic sum exists, then it coincides with the usual sum of operators.
\end{problem}}

The formulation of the problem allows to treat this representation either as a injective morphism  or an order reflecting morphism. 
The problem was first solved in \cite{niepa} for finite effect algebras, 
another  partial  answer  was obtained by H.X. Cao et al. in \cite{cao}
and in the full generality it was solved  for arbitrary effect algebras
by Rie\v canov\'a and Zajac in \cite{rieza}.

In what follows  we discuss both cases for generalized effect algebras, 
continuing in the same direction as in \cite{niepa}. Our, 
in some sense very compact, approach includes all the results 
mentioned above. 
After a  summary of the state of the art given in 
Section \ref{prelim}, the  detailed  presentation  of  our  new  results  is  split  in  Section \ref{oneone} and Section \ref{boundedop}.
The main result of Section \ref{oneone} 
are the necessary  and sufficient conditions of 
Problem \ref{problem} for a  
(generalized) effect algebra $E$ to be weakly representable  into  positive 
linear operators. Second, 
in Section \ref{boundedop} 
we establish the necessary  and sufficient conditions of 
Problem \ref{problem} for a  
(generalized) effect algebra $E$ to be representable into  positive 
linear operators.

As a by-product  of our study we establish topological properties of (weakly) representable generalized effect algebras.

\section{Preliminaries} 
\label{prelim}

Effect algebras were introduced by Foulis and 
 Bennett (see \cite{foulis}) 
for modelling unsharp measurements in a Hilbert space. In this case the 
set ${\mathcal E}({\mathcal H})$ of effects is the set of all self-adjoint operators $A$ on 
a Hilbert space ${\mathcal H}$ between the null operator $0$ and the identity 
operator $1$ and endowed with the partial operation $+$ defined 
iff  $A+B$ is in ${\mathcal E}({\mathcal H})$, 
where $+$ is the usual operator sum.  We call it a 
{\em Hilbert space effect algebra}.

In general form, an effect algebra is in fact a partial algebra 
with one partial binary operation and two unary operations satisfying 
the following axioms due to Foulis and 
 Bennett.

\begin{definition}{\rm { \rm (Foulis and Bennett, 1994, \cite{foulis})}
A partial algebra $(E;\oplus,0,1)$ is called an {\it effect algebra}
if 0,1 are two distinguish elements and $\oplus$ is a partially defined
binary operation on $E$ which satisfy the following conditions for
any $x,y,z\in E$:

(E1) $x\oplus y=y\oplus x$ if $x\oplus y$ is defined,

(E2) $(x\oplus y)\oplus z=x\oplus(y\oplus z)$ if one side is
defined,

(E3) for every $x\in E$ there exists a unique $y\in E$ such that
$x\oplus y=1$ (we put $x^\prime=y$),

(E4) If $1\oplus x$ is defined then $x=0$.}
\end{definition}

We often denote the effect algebra $(E;\oplus,0,1)$ briefly by $E$. 

Effect algebras are very natural algebraic 
structures for to be carriers of states or probability measures in the cases 
when elements are noncompatible or unsharp.

Generalizations of effect algebras (i.e., without a top element 1)
have been studied by K\^opka and Chovanec (1994) (difference
posets), Foulis and Bennett (1994) (cones), Kalmbach and Rie\v
canov\'a (1994) (abelian $RI$-posets and abelian $RI$ semigroups)
and Hedl\'ikov\'a and Pulmannov\'a (1996) (generalized $D$-posets
and cancelative positive partial abelian semigroups). It can be
shown that all of the above mentioned generalizations of effect
algebras are mutually equivalent. For recent results concerning 
them see e.g. \cite{pasekariec,polakovic,polriec,rieczapul}. 

\begin{definition}{\rm \advance\parindent by 5pt
\begin{enumerate}
\item[(1)] A {\it generalized effect algebra} ($E$; $\oplus$, 0) is a set $E$ with element $0\in E$ and partial
binary operation $\oplus$ satisfying for any $x,y,z\in E$ conditions
\begin{enumerate}
\item[(GE1)] $x\oplus y = y\oplus x$ if one side is defined,

\item[(GE2)] $(x\oplus y)\oplus z=x\oplus(y\oplus z)$ if one side is
defined,

\item[(GE3)] if $x\oplus y=x\oplus z$ then $y=z$,

\item[(GE4)] if $x\oplus y=0$ then $x=y=0$,

\item[(GE5)] $x\oplus 0=x$ for all $x\in E$.
\end{enumerate}
\item[(2)] A binary relation $\le$ (being a partial order) and 
a partial binary operation $\ominus$ on $E$ can be defined by:
\begin{center}
$x\le y$ \mbox{ and } {{$y\ominus x=z$}} \mbox{ iff }$x\oplus z$
\mbox{ is defined and }$x\oplus z=y$\,.
\end{center} 
\item[(3)] $Q\subseteq E$ is called a {\em sub-generalized effect algebra} 
({\em sub-effect algebra})
of $E$ iff out of elements $x,y,z\in E$ with $x\oplus y=z$ at least two are in $Q$ then $x,y,z\in
Q$ and $0\in Q$ ($1\in Q$ for sub-effect algebra).  
Then $Q$ is a generalized effect algebra (effect algebra) in its own right. 
\end{enumerate}}
\end{definition}

\begin{definition} {\rm Let ($E_1$; $\oplus_1$, $0_1$) and 
($E_2$; $\oplus_2$, $0_2$) be generalized effect algebras.}  {\rm \advance\parindent by 5pt
\begin{enumerate}
\item[(1)]  
A map $f : E_1 \to E_2$ is called a {\em morphism} if 
$f(a\oplus_1 b) = f(a) \oplus_2 f(b)$, for any $a, b\in E_1$
with defined $a \oplus_1 b$.
\item[(2)]   An injective morphism $f:E_1 \to  E_2$ such that 
$f(E_1)$ is a sub-generalized effect algebra of $E_2$
is called an {\em embedding} (or {\em monomorphism}).
\item[(3)]   A surjective monomorphism $f : E_1 \to E_2$ is 
called an {\em isomorphism}. 
\item[(4)]  A morphism $f:E_1 \to E_2$  is called  {\em order reflecting} 
if $f(a)\leq f(b)$ implies $a\leq b$ for all $a, b\in E_1$. 
\end{enumerate}}
\medskip

{\rm{}Any embedding is an order reflecting morphism and 
any  order reflecting morphism is  an injective morphism  but not conversely. However, 
any order reflecting morphism of effect algebras is an embedding (see  \cite[Proposition 1]{pasekariec2}).}

\end{definition}

Throughout the paper we assume that 
$\cH$ is a complex Hilbert space, i.e., a linear space with inner product
 $\langle\cdot\,,\cdot\rangle$ which is
complete in the induced metric. Recall that here for any 
$\vec{x},\vec{y}\in\cH$ we have 
$\langle \vec{x},\vec{y}\rangle\in\mathbb C$ 
(the set of complex numbers) such that 
$\langle \vec{x},\alpha \vec{y}+\beta \vec{z}\rangle=%
\alpha\langle \vec{x},\vec{y}\rangle+%
\beta\langle \vec{x},\vec{z}\rangle$ for all $\alpha,\beta\in\mathbb C$ and
$\vec{x},\vec{y},\vec{z}\in\cH$. Moreover, 
$\overline{\langle \vec{x},\vec{y}\rangle}=\langle \vec{y},\vec{x} \rangle$ 
and finally $\langle \vec{x},\vec{x}\rangle \ge 0$ at which 
$\langle \vec{x},\vec{x}\rangle=0$ iff $\vec{x}=0$.
The term {\it dimension} of $\cH$ in the following always means the {\it Hilbertian dimension} defined as the
cardinality of any orthonormal basis of $\cH$ (see \cite[p.~44]{blank}).

Moreover, we will assume that all considered linear operators $A$ (i.e., linear maps $A:D(A)\to\cH$) have
a domain $D(A)$ a linear subspace dense in $\cH$ with respect to metric topology induced by inner product,
so $\clos{D(A)}=\cH$ (we say that $A$ is {\em densely defined}).

A linear operator \(Q\) on a dense subspace \(D\)
of a Hilbert space is {\it positive} if, for all $\vec{x}\in D$,
\begin{center}
\(0\leq\soucin{\vec{x}}{Q(\vec{x})}\).
\end{center}

It is well-known that, for any set $M$, 
$$
l_2(M) = \{(x_m)_{m\in  M} \mid  x_m\in {\mathbb C}, 
\sum_{m\in M} |x_m|^{2} < \infty\}
$$
with the inner product $\langle (x_m)_{m\in  M}, (y_m)_{m\in  M}\rangle %
= \sum_{m\in M} \overline{x}_m y_m$  is a Hilbert space.  Recall that, 
for $(z_m)_{m\in  M}\in {\mathbb C}^{M}$ we have 
$\sum_{m\in M}  z_m=z\in {\mathbb C}$
if and only if for every $\varepsilon > 0$ there exists a 
finite $D_{\varepsilon} \subseteq M$ 
such that for every finite $G \subseteq M$ 
we have $D_{\varepsilon} \subseteq G \subseteq M$ 
$\Longrightarrow$ $|z- \sum_{m\in G}| < \varepsilon$. 
Hence the subspace 
${\mathcal E}_{lin}(M)=\{(x_m)_{m\in  M}\in l_2(M) \mid x_m=0 \ 
\text{for}\ \text{all}\ \text{but}\ \text{finitely}\ \text{many}\ 
m\in M\}$ is dense in $l_2(M)$. Note that  Kronecker's delta $\delta$ is a function of two variables which is $1$ if they are equal and $0$ otherwise. 
Clearly, ${\mathcal E}_{lin}(M)$ is a linear hull of the set 
$\{(\delta_{m, n})_{m\in M} \mid n\in M\}$. 

\begin{statement}%
\label{totoznost} 
\label{groupis} Let ${\cH}$ be a 
complex Hilbert space and  let $D\subseteq {\cH}$ be a linear subspace dense in ${\cH}$ 
(i.e. $\overline{D}={\cH}$). Then the following holds:
\begin{enumerate} 
\item[{\rm(1)}]{} {\rm{}(\cite[Theorem 1]{ptpaseka})} Let 
$$Lin_{D}({\cH})=\{A:D\to {\cH} \mid A\ \text{is a linear operator 
defined on}\ D\}.$$
Then $(Lin_{D}({\cH}); +, \leq, 0)$ is a partially ordered commutative group 
where $0$ is the null operator, $+$ is the usual sum of operators defined on $D$ and 
$\leq$ is defined for all $A, B\in Lin_{D}({\cH})$ 
by $A\leq B$ iff $B-A$ is positive.

\item[{\rm(2)}] {\rm{}([14], Theorem 3.1, see also [15])} Let 
$$
\begin{array}{r c l}
{\mathcal G}_D({\mathcal H})&=&%
\{A:D\to {\mathcal H} \mid A\ \text{is a positive linear operator 
defined on}\ D\}.
\end{array}$$
Then $({\mathcal G}_D({\mathcal H}); \oplus, 0)$ is a generalized effect algebra 
where $0$ is the null operator, $\oplus$ is the usual sum of operators defined on $D$. 
\end{enumerate}
\end{statement}

\begin{definition}\label{d26}{\rm 
Let $(E;\oplus,0)$ be a generalized effect algebra. Assume further
that ${\cH}$ is a complex Hilbert space and  
let $D\subseteq {\cH}$ be a linear subspace dense in ${\cH}$. We say that 
{\rm \advance\parindent by 5pt
\begin{enumerate}
\item[(1)]  $E$ is {\em weakly representable in positive linear operators} (shortly  {\em weakly representable})  
if  there is an injective
morphism $\varphi:E \to {\mathcal G}_D({\mathcal H})$.
\item[(2)]  $E$ is {\em representable in positive linear operators}  (shortly {\em representable})  if  there is an order reflecting 
morphism $\varphi:E \to {\mathcal G}_D({\mathcal H})$.
\end{enumerate}  }}
\end{definition}

Note that Definition \ref{d26} is a generalization of the definition of representable
effect algebras introduced in \cite{cao} and of the Hilbert space effect representation
introduced in \cite{rieza}.

\begin{definition}{\rm 
Let $(E;\oplus,0)$ be a generalized effect algebra. 
A map $s: E \to {\mathbb R}^{+}_0$
is called a {\em generalized state} if $s$ 
is a morphism (here ${\mathbb R}^{+}_0$
is assumed as a generalized effect algebra with the usual addition of real numbers). 

If moreover $E={\mathcal G}_D({\mathcal H})$ 
for some complex  Hilbert space ${\mathcal H}$ and 
$D\subseteq {\cH}$ a linear subspace dense in ${\cH}$, and 
$\vec{x}\in D$ 
we call the mapping 
$\omega_{\vec{x}}:{\mathcal G}_D({\mathcal H}) \to {\mathbb R}^{+}_0$ 
defined by the prescription 
$\omega_{\vec{x}}(A)={\soucin{\vec{x}}{A(\vec{x})}}$, 
$A\in {\mathcal G}_D({\mathcal H})$
a {\em generalized vector state}. Clearly, any 
 generalized vector state is a  generalized state on 
${\mathcal G}_D({\mathcal H}) $.

Let $S$ be a set of generalized states on $E$. Then 
 \begin{enumerate}
 \item[{\rm(i)}] $S$ {\em separates points} if 
 \begin{center}\(((\forall s \in S)\, s(a)= s(b))\implies a=b\)
\end{center}
for any elements \(a,b\in E\);
\item[{\rm(ii)}] $S$ is called {\em order determining} if 
 \begin{center}\(((\forall s \in S)\, s(a)\leq s(b))\implies a\leq
b\)\end{center}
for any elements \(a,b\in E\);
\item[{\rm(iii)}] $S$ is said to be {\em bounded} if, for any element \(a\in E\), there is $c_a>0$ such that 
 \begin{center}\((\forall s \in S)\, s(a)\leq c_a\); 
\end{center}
\item[{\rm(iv)}] we denote by $\tau_E^{S}$ the {\em  weak topology} 
(the topology generated by the family of mappings $S$ \cite[p. 31]{engelking}) {\em  on} $E$ 
{\em with respect to the set} $S$. If $E={\mathcal G}_D({\mathcal H})$ and 
$S$ is the set of all  generalized vector states on $E$ we use notation $\tau_D$ instead of $\tau_E^{S}$ .
\end{enumerate}

Let $(E;\oplus,0_E, 1_E)$ be an effect algebra and $s$ be  a  generalized state on $E$. Then 
$s$ is  
a {state} iff $s(1_E)= 1$.}
\end{definition}

Evidently, any order determining set of states separates points. 
Moreover, a set $S$ of generalized states on $E$ 
separates points iff the  map $i_S:E \to {({\mathbb R}^{+}_0)}^{S}$ defined by 
$i_S(a)=(s(a))_{s\in S}$ for all $a\in E$ is an injective
 morphism of generalized effect algebras. Similarly as 
in  \cite[Theorem 2.1]{gudder},  
 a set $S$ of states on an effect algebra $E$ 
separates points iff the  map $i_S:E \to {{[0, 1]}}^{S}$  
is an injective morphism of  effect algebras. 
In both cases the 
respective  map $i_S$  is continuous relative to  $\tau_E^{S}$ and the respective product topology.

Note also that any non-trivial 
generalized state $g:E\to {\mathbb R}^{+}_0$ 
on an effect algebra $E$ yields a state $s:E\to [0, 1]$ 
given by $s(y)=\frac{g(y)}{g(1)}$ for all $y\in E$.

\begin{example} \label{excd} {\rm Let 
${\mathcal H}={\mathbb C}^{2}$, $\pi_1(x, y)=(x,0)$ 
and $\pi_2(x, y)=(0,y)$ for all $(x, y)\in {\mathbb C}^{2}$. Then the set 
$E=\{0, \pi_1, \pi_2\}$ equipped 
with the commutative partial operation $\oplus$ given 
by $0=0\oplus 0$, $0 \oplus \pi_1=\pi_1$ and $0 \oplus \pi_2=\pi_2$ is 
a generalized effect algebra with an order determining set of generalized 
states induced from ${\mathcal E}({\mathcal H})$ but it is not 
a sub-effect algebra of ${\mathcal E}({\mathcal H})$. Namely, 
$\pi_1\oplus_{{\mathcal E}({\mathcal H})} \pi_2=\hbox{id}_{{\mathcal H}}$ 
but $\hbox{id}_{{\mathcal H}}\notin E$. Moreover, 
the  inclusion map $i:E \to {\mathcal E}({\mathcal H})$ is an order 
reflecting morphism and it is not an embedding.}
\end{example}

Suppose we are given a topological space $X$, a family 
$\{ Y_s\}_{s\in S}$ of topological spaces and
a family of continuous maps 
${\mathcal F } = \{f_s\}_{s\in S}$ where  
$f_s: X \to Y_s$.  If for every 
$x \in X$ and every closed set $G \subseteq X$ such that 
$x \not\in G$ there exists a map 
$f_s\in {\mathcal F}$  such that 
$f_s(x) \not\in \overline{f_s(F)}$, then we say that the {\em family }
$\mathcal F$ {\em separates points and closed sets} (see \cite[p. 82]{engelking}).

More details on linear operators on Hilbert spaces  with many 
examples and contraexamples can be found in \cite{blank}, on topology 
in \cite{engelking}  
and about (generalized) effect algebras in \cite{dvurec}.

   \section{Injective morphisms and generalized states}
   \label{oneone}

We are going to show that a  
(generalized) effect algebra $E$ is weakly representable in  positive 
linear operators  if and only if it has 
a set $S$ of generalized states such that $S$ separates points. In 
addition, any generalized state from $S$ is induced by a suitable 
generalized vector state.

\begin{theorem}\label{umthex}
Let $E$ be a  generalized 
effect algebra \(E\) and let $S$ be a set of generalized states 
on $E$. Then there exists a  morphism 
$\varphi_S:E \to {\mathcal G}_{{\mathcal E}_{lin}(S)}(l_2(S))$ such that, 
for all $s\in S$, there is a vector 
$\vec{x}_s\in {l_2(S)}$  satisfying $s=\omega_{\vec{x}_s}\circ \varphi$. 

Moreover, 
\begin{enumerate}
\item[{\rm (i)}] $\varphi_S:(E,\tau_E^{S})\to ({\mathcal G}_{{\mathcal E}_{lin}(S)}(l_2(S)), \tau_{{\mathcal E}_{lin}(S)})$ is 
a continuous map;%
\item[{\rm (ii)}] there is an index set $T$ and 
a family ${\mathcal F } = \{f_t\}_{t\in T}$ of continuous maps 
that separates points and closed sets where  
$f_t: E \to [0,1]$;  
\item[{\rm (iii)}] if  $S$ separates points then $\varphi_S$ is injective and $\tau_E^{S}$ is Tychonoff;
\item[{\rm (iv)}]  if  $S$ is order determining then $\varphi_S$ is order reflecting ;
\item[{\rm(v)}] if $S$ is  {bounded} then, for any element \(a\in E\), $\varphi_S(a)$ is a bounded operator.
\end{enumerate}
\end{theorem}
\begin{proof}  
We define 
$\varphi_S: E \to {\mathcal G}_{{\mathcal E}_{lin}(S)}(l_2(S))$ as follows: 
For any $a\in E$ and ${\mathbf x}=(x_s)_{s\in  S}\in {\mathcal E}_{lin}(S)$, 
we set 
$$\varphi_S(a)({\mathbf x})=(s(a)x_s)_{s\in  S}\in {\mathcal E}_{lin}(S).$$

Then clearly $ \varphi_S(a)$ is a linear operator from ${\mathcal E}_{lin}(S)$ into $l_2(S)$ and 
 the set $\text{supp}({\mathbf x})=\{s\in S \mid x_s\not=0\}$ is finite. Hence 
$\soucin{{\mathbf x}}{\varphi_S(a)({\mathbf x})}=\sum_{s\in \text{supp}({\mathbf x})}s(a)x_s^2\geq 0$ yields that 
$\varphi_S(a)\in {\mathcal G}_{{\mathcal E}_{lin}(S)}(l_2(S))$.

Evidently, for all ${\mathbf x}\in {\mathcal E}_{lin}(S)$, $\varphi_S(0)({\mathbf x})=(s(0)x_s)_{s\in  S}=%
(0)_{s\in  S}={\mathbf 0}$. Hence $\varphi_S(0)=0_{{\mathcal G}(S)}$. 

Now, let $a, b\in E$ such that $a\oplus b$ exists. Then, 
for all ${\mathbf x}\in {\mathcal E}_{lin}(S)$, 
$\varphi_S(a\oplus b)({\mathbf x})=(s(a\oplus b)x_s)_{s\in  S}%
=((s(a)+ s(b))x_s)_{s\in  S}=%
(s(a)x_s)_{s\in  S}+(s(b)x_s)_{s\in  S}=%
\varphi_S(a)({\mathbf x})+\varphi_S(b)({\mathbf x})$. 
Therefore $\varphi_S(a\oplus b)=\varphi_S(a)\oplus \varphi_S(b)$.


For any $s \in S$ and any ${\mathbf x}_{s}=(\delta_{s,t})_{t\in S}$, 
we have to  compute, for all $a\in E$,  
 $(\omega_{\vec{x}_s}\circ \varphi_S)(a)$. We have  
 $$(\omega_{\vec{x}_s}\circ \varphi_S)(a)=%
\soucin{\vec{x}_s}{\varphi_S(a)(\vec{x}_s)}=%
\soucin{(\delta_{s,t})_{t\in S}}{(t(a)\delta_{s,t})_{t\in S}}=%
s(a).$$

(i) Denote by $\tau_{{\mathbb R}^{+}_0}$ the euclidean topology on ${\mathbb R}^{+}_0$. The set 
${\mathcal B}=\bigcup_{s\in S} \omega_{\vec{x}_s}^{-1}(\tau_{{\mathbb R}^{+}_0}) $ is a subbase of $\tau_{{\mathcal E}_{lin}(S)}$ and the 
 set 
${\mathcal C}=\bigcup_{s\in S} s^{-1}(\tau_{{\mathbb R}^{+}_0})%
=\bigcup_{s\in S} (\omega_{\vec{x}_s}\circ \varphi_S)^{-1}(\tau_{{\mathbb R}^{+}_0})= %
 \varphi_S^{-1}({\mathcal B})$ is a subbase of $\tau_{E}^{S}$. It follows that $\varphi_S$ is a continuous map. 


(ii) Let us put $T=\{(a, G) \mid %
a\in E, G\  \text{is a closed set in}\  \tau_{E}^{S}, a\not\in G\}$.
 Now, let $a\in E$ and let $G$ be a closed set in $\tau_{E}^{S}$ such that $a\not\in G$. Since ${\mathcal C}$ is 
a subbase of $\tau_{E}^{S}$ we get that there are $s_1, \dots, s_n\in S$ and closed sets $G_1, \dots, G_n$ 
in $ \tau_{{\mathbb R}^{+}_0}$ such that 
$a\not\in \bigcup_{i=1}^{n} s_i^{-1}( G_{i})$, $G\subseteq \bigcup_{i=1}^{n} s_i^{-1}( G_{i})$. 
Then, for any $i=1, \dots, n$, $s_i(a)\not\in G_i$. 
We have a continuous map $(s_1, \dots, s_n):E\to ({\mathbb R}^{+}_0)^{n}$ given by the prescription 
$(s_1, \dots, s_n)(x)=(s_1)x), \dots, s_n(x))$ for all $x\in E$. Let 
us put $H=\bigcup_{i=1}^{n} {\pi}_i^{-1}( G_{i})$ where 
$ {\pi}_i: ({\mathbb R}^{+}_0)^{n}\to {\mathbb R}^{+}_0$ are the respective projections. 
Then $H$ is closed in the product topology on $ ({\mathbb R}^{+}_0)^{n}$ which is Tychonoff and $(s_1, \dots, s_n)(a)\notin H$. 
Hence there is a continuous map $g: ({\mathbb R}^{+}_0)^{n} \to [0,1]$ such 
that $g((s_1, \dots, s_n)(a)) = 0$ and $g(y) = 1$ for $y \in H$. Let us put 
$f_{(a, G)}=g\circ (s_1,\dots, s_n)$. It follows that $f_{(a, G)}:E \to [0,1]$ is a continuous function such that 
$f_{(a, G)}(a)=0$ and $f_{(a,G)}(G)\subseteq \{1\}$.

(iii) Suppose that  $S$ separates points.  Let us check that $\varphi_S$ is injective. Assume that 
$a, b\in E$ and $\varphi_S(a)=\varphi_S(b)$. 
Then, for any $s \in S$ and any ${\mathbf x}_{s}=(\delta_{s,t})_{t\in S}$, 
we have that $s(a)=\varphi_S(a)({\mathbf x}_{s})%
=\varphi_S(b)({\mathbf x}_{s})=s(b)$. Since $S$ separates points  
 we obtain that $a= b$. 

First, let us show that $\tau_{E}^{S}$ is Hausdorff.  Let $a, b\in E$ such that $a\not =b$. Then there is $s_0\in S$ such that 
 $s_0(a)\not =s_0(b)$. Since $\tau_{{\mathbb R}^{+}_0}$ is Hausdorff there are disjoint open sets 
$U, V\in \tau_{{\mathbb R}^{+}_0}$ such that 
$s_0(a)\in U$,  $s_0(b)\in V$. Hence $a\in s_0^{-1}(U)$ and $b\in s_0^{-1}(V)$, 
$s_0^{-1}(U)$ and $s_0^{-1}(V)$ are disjoint open sets in $\tau_{E}^{S}$.

Now, let $a\in E$ and let $G$ be a closed set in $\tau_{E}^{S}$ such that $a\not\in G$. By (ii) there is 
a continuous function  $f_{(a, G)}:E \to [0,1]$  such that 
$f_{(a, G)}(a)=0$ and $f_{(a, G)}(G)\subseteq \{1\}$. Hence $\tau_{E}^{S}$ is Tychonoff.

(iv) Suppose that  $S$  is order determining.  Let us  check that $\varphi$ is 
order reflecting. Since $\varphi_S$ preserves order  we have to 
check that $a, b\in E$, $\varphi_S(a)\leq \varphi_S(b)$ yields $a\leq b$. 
But, for any $s\in S$, we have by the positivity of 
$\varphi_S(b)- \varphi_S(a)$ that 
$$
\langle (\delta_{s, t})_{t\in S}, %
(\varphi_S(b)- \varphi_S(a))(\delta_{s, t})_{t\in S}\rangle = s(b)-s(a)\geq 0.
$$
Since $S$ is order determining we get $a\leq b$.

(v) Suppose that  $S$  is bounded and let $a\in E$. We have to show that $\varphi_S(a)$ is a bounded operator. 
We have, for any   ${\mathbf x}=(x_s)_{s\in  S}\in {\mathcal E}_{lin}(S)$, 
$$
|| \varphi_S(a)({\mathbf x})||=\sqrt{\sum_{s\in \text{supp}({\mathbf x})}s(a)^{2}x_s^2}%
\leq \sqrt{\sum_{s\in \text{supp}({\mathbf x})}c_a^{2}x_s^2}=%
c_a\sqrt{\sum_{s\in \text{supp}({\mathbf x})}x_s^2}=c_a || {\mathbf x}||.
$$
\hfill
\end{proof}

\begin{remark}\label{remx}{\rm Recall that, for an at  most countable countable set $S$ of generalized states on $E$ and any 
$a\in E$, the operator $\varphi_S(a)$ coincides with the operator $\dot{T}_{(s(a))_{s\in S}}$ introduced in 
\cite[Example 4.1.4 and Example 4.2.2]{blank}. Using the same reasonings as in \cite{blank} one can easily show that 
$\varphi_S(a)$ is actually essentially self-adjoint \cite[p. 96]{blank}.}
\end{remark}

\begin{theorem}\label{mthex}
The following conditions are equivalent for every generalized 
effect algebra \(E\).\\
{\rm(i)}
There exists a set $S$ of generalized states 
on $E$ that separates points;\\
{\rm(ii)}
there exists an injective morphism 
$\varphi:E \to {\mathcal G}_D({\mathcal H})$ for some 
complex  Hilbert space ${\mathcal H}$ and 
$D\subseteq {\cH}$ a linear subspace dense in ${\cH}$;\\ 
{\rm(iii)} there exists a set $T$ and an injective 
morphism of   generalized effect algebras from  \(E\) to  ${({\mathbb R}^{+}_0)}^{T}$. \\
\end{theorem}
\begin{proof} (i) \(\Longleftrightarrow\) (iii) It is evident. \\
(i) \(\implies\) (ii) 
It follows from Theorem \ref{umthex} (iii).\\
(ii) \(\implies\) (i)
Let $\varphi: E \to ({\mathcal G}_D({\mathcal H}); \oplus, 0)$ 
be an injective morphism for some 
complex  Hilbert space ${\mathcal H}$ and 
$D\subseteq {\cH}$ a linear subspace dense in ${\cH}$.  
Let \(a,b\in E\) and \(a\ne b\). By assumption, \(\varphi(a)\ne \varphi(b)\).
\\
Assume first \(\soucin{\vec{h}}{\varphi(a)(\vec{h})}=%
\soucin{\vec{h}}{\varphi(b)(\vec{h})}\) for
every \(\vec{h}\). According to Statement \ref{totoznost},
the situation cannot occur.
Hence there exists \(\vec{h_0}\in \cH\) for which  
\(\soucin{\vec{h_0}}{\varphi(a)(\vec{h_0})}\ne
\soucin{\vec{h_0}}{\varphi(b)(\vec{h_0})}\). Then 
$s_{a,b}:E\to {\mathbb R}^{+}_0$ given by the prescription 
$y\mapsto
{\soucin{\vec{h_0}}{\varphi(y)(\vec{h_0})}}$ for all $y\in E$
 is a generalized state for which \(s_{ab}(a)\ne s_{ab}(b)\). 
 
 It follows that the set 
$S=\{\omega_{\vec{x}}\circ \varphi \mid \vec{x}\in \cH\}$ separates points.
\end{proof}

\begin{corollary}
\label{cdmujcorrieyapul} Let ${\cH}$ be a 
complex Hilbert space 
and let $D\subseteq {\cH}$ be 
a linear subspace dense in ${\cH}$.  Then 
the topological space 
$({{\mathcal G}_D({\mathcal H})},\tau_D)$ is Tychonoff. 
\end{corollary}

\begin{corollary}The following conditions and the conditions {\rm(i)-(iii) }
from Theorem \ref{mthex} are equivalent for every 
effect algebra \(E\).\\
{\rm(a)}
There exists a set  $S$ of  states on $E$ that separates points;\\
{\rm(b)} there exists a complex Hilbert space ${\mathcal H}$ and 
an injective morphism of effect algebras from \(E\) to the 
 Hilbert space effect algebra ${\mathcal E}({\mathcal H})$;\\ 
{\rm(c)} there exists a set $T$ and an injective morphism of 
effect algebras from \(E\) to  ${{[0,1]}}^{T}$. 
\end{corollary}
\begin{proof} Evidently, (i)  \(\implies\) (a), (b)  \(\implies\) (ii), 
(ii) \(\implies\) (i) and 
(a) \(\Longleftrightarrow\) (c). Let us show that (a)  \(\implies\) (b).
As in \cite[Remark 3]{rieza} we have an injective 
morphism of effect algebras from \(E\) to the 
 Hilbert space effect algebra ${\mathcal E}({ l_2(S)})$.
\end{proof}

   \section{Order reflecting morphisms and generalized states}
   \label{boundedop}
\noindent

We are going to prove that a  
(generalized) effect algebra $E$ is representable in  positive 
linear operators  if and only if it has 
an order determining  set $S$ of generalized states.

Note that a set $S$ of generalized states on $E$ is 
order determining iff the  morphism $i_S:E \to {({\mathbb R}^{+}_0)}^{S}$ is order reflecting. As in  \cite[Theorem 2.1]{gudder},  
 a set $S$ of states on an effect algebra $E$ is 
order determining iff the  morphism 
$i_S:E \to {{[0, 1]}}^{S}$  
is an  embedding of  effect algebras.

\begin{theorem}\label{vnoreni}
For every generalized effect algebra \(E\), the 
following conditions are equivalent.
\\
{\rm(i)} There exists an order determining set 
\(S\) of  generalized  states on $E$;
\\
{\rm(ii)} there exists an order reflecting morphism 
$\varphi: E \to {\mathcal G}_D({\mathcal H})$ for some 
complex  Hilbert space ${\mathcal H}$ and 
$D\subseteq {\cH}$ a linear subspace dense in ${\cH}$;\\  
{\rm(iii)} there exists a set $T$ and an order reflecting morphism  from
 \(E\) to  ${({\mathbb R}^{+}_0)}^{T}$. \\

\end{theorem}
\begin{proof} (i) \(\Longleftrightarrow\) (iii) It is evident. \\
(i) \(\implies\) (ii) It follows from Theorem \ref{umthex} (iv).\\
(ii) \(\implies\) (i) By \cite[Theorem 5.3]{niepa}. Note only that the 
set 
$S=\{\omega_{\vec{x}}\circ \varphi \mid \vec{x}\in \cH\}$ is order 
determining.
\end{proof}

\begin{corollary}
\label{mujcorrieyapul} Let ${\cH}$ be a 
complex Hilbert space 
and let $D\subseteq {\cH}$ be 
a linear subspace dense in ${\cH}$.  Then 
the generalized effect algebra 
${{\mathcal G}_D({\mathcal H})}$ possesses 
an order determining set of generalized states. 
\end{corollary}

\begin{corollary}{\rm\cite[Corollary 2]{rieza}} 
\label{corrieya} The following conditions and the conditions {\rm(i)-(iii) }
from Theorem \ref{vnoreni} are equivalent for every 
effect algebra \(E\).\\
{\rm(a)}
There exists an order determining  set 
\(S\) of states on $E$;
\\
{\rm(b)} there exists a complex Hilbert space ${\mathcal H}$ and 
an  embedding  of effect algebras from  \(E\) to the 
 Hilbert space effect algebra ${\mathcal E}({\mathcal H})$;\\ 
{\rm(c)} there exists a set $T$ and an  embedding  of 
effect algebras from \(E\) to  ${{[0,1]}}^{T}$. 
\end{corollary}
\begin{proof} Evidently, (i)  \(\implies\) (a), (b)  \(\implies\) (ii), 
(ii) \(\implies\) (i) and 
(a) \(\Longleftrightarrow\) (c). Let us show that (a)  \(\implies\) (b).
We may use either \cite[Corollary 2]{rieza} or observe 
that 
any $\varphi_S(a)$, $a\in E$, is bounded by Theorem \ref{umthex} (v) and hence 
it posseses a unique extension to $l_2(S)$. Moreover, 
it is obvious that $\varphi_S(a)\leq I$.
Therefore  we have an 
 embedding of effect algebras from \(E\) to the 
 Hilbert space effect algebra ${\mathcal E}({ l_2(S)})$.
\end{proof}

\begin{remark} {\rm Note that by \cite[Lemma 3.3]{niepa} the case (b) from 
Corollary \ref{corrieya} is equivalent with the statement that $E$ is 
isomorphic to a sub-effect algebra of the 
 Hilbert space effect algebra ${\mathcal E}({\mathcal H})$, a result established in \cite{rieza}. In the generalized 
 algebra case 
(see Example \ref{excd}) we have to work with order reflecting maps instead of embeddings.
}\end{remark}



\begin{thebibliography}{99}

\bibitem{blank} J. Blank, P. Exner, M. Havl\'\i\v cek,  Hilbert
Space Operators in Quantum Physics (Second edition),  Springer,
2008.


\bibitem{cao} H.X. Cao, Z.H. Guo, Z.L. Chen, K.L. Zhang,  
Representation Theory of Effect Algebras,  (2012),
preprint,

\bibitem{dvurec} A. Dvure\v censkij, S. Pulmannov\'a, New trends
in quantum structures,  Kluwer,  Dodrecht, The
Netherlands, 2000.

\bibitem{engelking} R. Engelking, General topology, revised and completed edition, Heldermann Verlag, Berlin, 1989.

\bibitem{foulis} D.J. Foulis, M.K. Bennett, Effect algebras 
and unsharp quantum logics, Found. Phys. 24 (1994), 1331--1352. 


\bibitem{greechie} R. J. Greechie, Another Nonstandard Quantum Logic (And How I Found It), in Mathematical Foundations of Quantum Theory (Papers from a conference held at Loyola University, New Orleans, June 2-4, 1977), edited by A. R. Marlow, Academic Press, New York, 1978, 71--85. 

\bibitem{gudder} S. Gudder, Effect Algebras Are Not Adequate Models for 
Quantum Mechanics, Found. Phys. 40 (2010), 1566--1577. 

\bibitem{6}  J. Hedl\'\i kov\'a, S. Pulmannov\'a,  Generalized
difference posets and orthoalgebras, Acta Math. Univ.
Comenianae LXV (1996), 247--279.

\bibitem{7} G. Kalmbach, Z. Rie\v canov\'a,  An
axiomatization for abelian relative inverses, Demonstratio
Math. 27 (1994), 769--780.

\bibitem{8}  F. K\^opka, F. Chovanec,  D-posets, Math.
Slovaca 44 (1994), 21--34.


\bibitem{niepa}  J. Niederle, J. Paseka, 
{On realization of generalized effect algebras}, (2012), preprint. 

\bibitem{ptpaseka}J. Paseka, {PT-Symmetry in (Generalized) Effect Algebras}, 
{\it Internat. J. Theoret. Phys.}, {\bf 50} (2011), 1198--1205. 

\bibitem{pasekariec} J. Paseka, Z. Rie\v{c}anov\'{a},
{Considerable Sets of Linear Operators in Hilbert
Spaces as Operator Generalized Effect Algebras}, 
{Foundations of Physics}, (10), 41 (2011), 1634--1647, 
 doi: 10.1007/s10701-011-9573-0. 


\bibitem{pasekariec2} J. Paseka, Z. Rie\v{c}anov\'{a},
Inherited properties of effect algebras 
preserved by isomorphisms, preprint, 2012.

\bibitem{polakovic} M. Polakovi\v c, Generalized Effect Algebras 
of Bounded Positive Operators Defined on Hilbert Spaces,
Reports on Mathematical Physics, 68 (2011), 241--250.

\bibitem{polriec} M. Polakovi\v c, Z. Rie\v canov\'a,  Generalized 
effect algebras of positive operators densely defined on Hilbert space,
Internat. J. Theor. Phys.,   { 50}  (2011), 1167--1174, 
doi: 10.1007/s10773-010-0458-3.

\bibitem{rieczapul} Z. Rie\v canov\'a,   M.  Zajac, 
S. Pulmannov\'a, Effect algebras of positive operators densely 
defined on Hilbert space, Reports on Mathematical Physics, 
   { 68}  (2011), 261--270.

\bibitem{rieza} Z. Rie\v canov\'a,  M. Zajac, 
Hilbert space effect-representations of effect
algebras,  Reports on Mathematical Physics,  accepted. 



\end{thebibliography}
\end{document}